\documentclass[11pt]{amsart}
\usepackage{pdfsync}
\usepackage{latexsym}
\usepackage{amsgen,amsmath,amsxtra,amssymb,amsfonts,amsthm}
\usepackage{times}
\usepackage{graphicx}
\usepackage[dvips]{epsfig}
\usepackage{subfigure}
\usepackage[active]{srcltx}
\usepackage{color}

\begin{document}
\newcommand{\dyle}{\displaystyle}
\newcommand{\R}{{\mathbb{R}}}
 \newcommand{\Hi}{{\mathbb H}}
\newcommand{\Ss}{{\mathbb S}}
\newcommand{\N}{{\mathbb N}}
\newcommand{\Rn}{{\mathbb{R}^n}}
\newcommand{\ieq}{\begin{equation}}
\newcommand{\eeq}{\end{equation}}
\newcommand{\ieqa}{\begin{eqnarray}}
\newcommand{\eeqa}{\end{eqnarray}}
\newcommand{\ieqas}{\begin{eqnarray*}}
\newcommand{\eeqas}{\end{eqnarray*}}
\newcommand{\Bo}{\put(260,0){\rule{2mm}{2mm}}\\}
\def\L#1{\label{#1}} \def\R#1{{\rm (\ref{#1})}}


\theoremstyle{plain}
\newtheorem{theorem}{Theorem} [section]
\newtheorem{corollary}[theorem]{Corollary}
\newtheorem{lemma}[theorem]{Lemma}
\newtheorem{proposition}[theorem]{Proposition}
\def\neweq#1{\begin{equation}\label{#1}}
\def\endeq{\end{equation}}
\def\eq#1{(\ref{#1})}

\theoremstyle{definition}
\newtheorem{definition}[theorem]{Definition}
\newtheorem{remark}[theorem]{Remark}

\numberwithin{figure}{section}
\newcommand{\res}{\mathop{\hbox{\vrule height 7pt width .5pt depth
0pt \vrule height .5pt width 6pt depth 0pt}}\nolimits}
\def\at#1{{\bf #1}: } \def\att#1#2{{\bf #1}, {\bf #2}: }
\def\attt#1#2#3{{\bf #1}, {\bf #2}, {\bf #3}: } \def\atttt#1#2#3#4{{\bf #1}, {\bf #2}, {\bf #3},{\bf #4}: }
\def\aug#1#2{\frac{\displaystyle #1}{\displaystyle #2}} \def\figura#1#2{ \begin{figure}[ht] \vspace{#1} \caption{#2}
\end{figure}} \def\B#1{\bibitem{#1}} \def\q{\int_{\Omega^\sharp}}
\def\z{\int_{B_{\bar{\rho}}}\underline{\nu}\nabla (w+K_{c})\cdot
\nabla h} \def\a{\int_{B_{\bar{\rho}}}}
\def\b{\cdot\aug{x}{\|x\|}}
\def\n{\underline{\nu}} \def\d{\int_{B_{r}}}
\def\e{\int_{B_{\rho_{j}}}} \def\LL{{\mathcal L}}
\def\itr{\mathrm{Int}\,}
\def\D{{\mathcal D}}
 \def\tg{\tilde{g}}
\def\A{{\mathcal A}}
\def\S{{\mathcal S}}
\def\H{{\mathcal H}}
\def\M{{\mathcal M}}
\def\T{{\mathcal T}}
\def\U{{\mathcal U}}
\def\N{{\mathcal N}}
\def\I{{\mathcal I}}
\def\F{{\mathcal F}}
\def\J{{\mathcal J}}
\def\E{{\mathcal E}}
\def\F{{\mathcal F}}
\def\G{{\mathcal G}}
\def\HH{{\mathcal H}}
\def\W{{\mathcal W}}
\def\H{\D^{2*}_{X}}
\def\d{d^X_M }
\def\LL{{\mathcal L}}
\def\H{{\mathcal H}}
\def\HH{{\mathcal H}}
\def\itr{\mathrm{Int}\,}
\def\vah{\mbox{var}_\Hi}
\def\vahh{\mbox{var}_\Hi^1}
\def\vax{\mbox{var}_X^1}
\def\va{\mbox{var}}
\def\SS{{\mathcal S}}
 \def\Y{{\mathcal Y}}
\def\length{{l_\Hi}}
\newcommand{\average}{{\mathchoice {\kern1ex\vcenter{\hrule
height.4pt width 6pt depth0pt} \kern-11pt} {\kern1ex\vcenter{\hrule height.4pt width 4.3pt depth0pt} \kern-7pt} {} {} }}
\def\weak{\rightharpoonup}
\def\detu{{\rm det}(D^2u)}
\def\detut{{\rm det}(D^2u(t))}
\def\detvt{{\rm det}(D^2v(t))}
\def\detv{{\rm det}(D^2v)}
\def\uuu{u_xu_yu_{xy}}
\def\uuut{u_x(t)u_y(t)u_{xy}(t)}
\def\uuus{u_x(s)u_y(s)u_{xy}(s)}
\def\uuutn{u_x(t_n)u_y(t_n)u_{xy}(t_n)}
\def\vvv{v_xv_yv_{xy}}
\newcommand{\ave}{\average\int}

\title[Radial biharmonic $k-$Hessian equations]{Existence of radial solutions to biharmonic $k-$Hessian equations}

\author[C. Escudero, P. J. Torres]{Carlos Escudero, Pedro J. Torres}
\address{}
\email{}
\keywords{$k-$Hessian type equations, Biharmonic boundary value problems, Existence of solutions, Non-existence of solutions.
\\ \indent 2010 {\it MSC: 34B08, 34B16, 34B40, 35G30.}}

\date{\today}

\begin{abstract}
This work presents the construction of the existence theory of radial solutions to the elliptic equation
\begin{equation}\nonumber
\Delta^2 u = (-1)^k S_k[u] + \lambda f(x), \qquad x \in B_1(0) \subset \mathbb{R}^N,
\end{equation}
provided either with Dirichlet boundary conditions
\begin{eqnarray}\nonumber
u = \partial_n u = 0, \qquad x \in \partial B_1(0),
\end{eqnarray}
or Navier boundary conditions
\begin{equation}\nonumber
u = \Delta u = 0, \qquad x \in \partial B_1(0),
\end{equation}
where the $k-$Hessian $S_k[u]$ is the $k^{\mathrm{th}}$ elementary symmetric polynomial of eigenvalues of the Hessian matrix
and the datum $f \in L^1(B_1(0))$ while $\lambda \in \mathbb{R}$. We prove the existence of a Carath\'eodory solution to these boundary value problems that is unique in a certain neighborhood of the origin provided $|\lambda|$ is small enough. Moreover, we prove that the solvability set of $\lambda$ is finite, giving an explicity bound of the extreme value.
\end{abstract}
\renewcommand{\thefootnote}{\fnsymbol{footnote}}
\setcounter{footnote}{-1}
\footnote{Supported by MTM2010-18128, RYC-2011-09025.}
\renewcommand{\thefootnote}{\arabic{footnote}}
\maketitle

\section{Introduction}

This work is devoted to the study of the existence of radial solutions to elliptic equations of the form
\begin{equation}\label{rkhessian}
\Delta^2 u = (-1)^k S_k[u] + \lambda f(x), \qquad x \in B_1(0) \subset \mathbb{R}^N,
\end{equation}
where $N, \, k \, \in \mathbb{N}$, $\lambda \in \mathbb{R}$ and $f: B_1(0) \subset \mathbb{R}^N \longrightarrow \mathbb{R}$ is an absolutely
integrable function. The first term in the right hand side of~\eqref{rkhessian} is the $k-$Hessian $S_k[u]=\sigma_k(\Lambda)$, where
$$
\sigma_k(\Lambda)= \sum_{i_1<\cdots<i_k} \Lambda_{i_1} \cdots \Lambda_{i_k},
$$
is the $k^{\mathrm{th}}$ elementary symmetric polynomial and $\Lambda=(\Lambda_1,\cdots,\Lambda_n)$ is the set of eigenvalues of the Hessian matrix $(D^2 u)$.
In other words, $S_k[u]$ is the sum of the $k^{\mathrm{th}}$ principal minors of the Hessian matrix.
We will always focus on the range $2 \le k \le N$, since equation~\eqref{rkhessian} is linear for $k=1$,
and we are interested in nonlinear boundary value problems.

The motivation to study equation~\eqref{rkhessian} comes from different sources.
In the first place we can cite the impressive development of analytical results concerning the fully nonlinear boundary value problems
$$
S_k[u] = f,
$$
as well as related problems, that has appeared in the last
decades~\cite{caffarelli3,wang1,wang2,labutin,wang3,trudinger,trudinger1,wang4,wang5,wang6,wang7,wang8,wang9,wang10,wang11,wang12,wang}.
Particular cases of the $k-$Hessian equation include the Poisson equation
$$
-\Delta u = f,
$$
for $k=1$, and the Monge-Amp\`ere equation~\cite{caffarelli1,caffarelli2}
$$
\det(D^2 u) = f,
$$
for $k=N$.

This work is also motivated by the theory of biharmonic boundary value problems.
Although they have been studied much less frequently than their harmonic equivalents,
they are present in many different applications and possess an inherent theoretical interest.
The current knowledge of fourth order elliptic equations has considerably grown in recent times~\cite{GGS},
but still it is not comparable to the stage of development of the theory concerning
harmonic boundary value problems. Biharmonic boundary value problems studied so far include different nonlinearities, see for instance~\cite{AGGM,BG,CEGM,DDGM,DFG,FG,FGK,moradifam}.
Nevertheless, to our knowledge, Hessian nonlinearities were considered for the first time in~\cite{n5}.
On one hand, it is natural to consider fourth order equations with
nonlinearities that involve the second derivatives of the solution.
And in turn it is natural to consider for these nonlinearities the $k-$Hessians,
since the Hessian matrix has exactly $N$ tensorial invariants, which are
the $N$ different $k-$Hessians. Consequently one of our objectives is to push forward the existence theory that concerns
this type of problems~\cite{n1,n2,n3,n4,n5}, which, as discussed, arises rather naturally within the theory of fourth order boundary value problems.

We can still mention one source of inspiration more. It is the presence
of these equations in the fields of condensed matter and statistical physics~\cite{escudero,escudero2}.
Although the present work is devoted to the construction of mathematical theory rather than modeling or the exploration of new applications,
we will briefly mention some potential implications of our results in physics in our last section.

To be concrete, we will concentrate on the biharmonic boundary value problem
\begin{eqnarray}\label{dirichlet}
\Delta^2 u = (-1)^k S_k[u] + \lambda f(x), \qquad x &\in& B_1(0) \subset \mathbb{R}^N, \\ \nonumber
u = \partial_n u = 0,
\qquad x &\in& \partial B_1(0),
\end{eqnarray}
which we denote as the Dirichlet problem for partial differential equation~\eqref{rkhessian}, and also
\begin{eqnarray}\label{navier}
\Delta^2 u = (-1)^k S_k[u] + \lambda f(x), \qquad x &\in& B_1(0) \subset \mathbb{R}^N, \\ \nonumber
u = \Delta u = 0,
\qquad x &\in& \partial B_1(0),
\end{eqnarray}
which we denote as the Navier problem for partial differential equation~\eqref{rkhessian}.

The paper is organized as follows. In Section~2, we write these problems in radial coordinates, then a change of variables leads to a Duffing equation with boundary conditions on the semi-infinite interval $[0,+\infty)$ that is more amenable to mathematical analysis. We moreover present our main results there. In Section~3 we develop the existence and local uniqueness theory of weak solutions for small $|\lambda|$, by applying the contraction principle. In Section~4, we prove higher regularity for these solutions. In Section~5, we use the upper and lower solution method to derive an explicit lower bound of the supremum of the solvability set of parameters $\lambda$ for which the problem is solvable. In Section~6 we show that no solutions exist for large $\lambda$. Finally, in Section~7, we will draw some conclusions that are implied by our results.

\section{Radial problems}\label{radialp}

The radial problem corresponding to equation~\eqref{rkhessian} reads
\begin{equation}\label{radialhessian}
\frac{1}{r^{N-1}} [r^{N-1} (\Delta_r u)']' = \frac{(-1)^k}{k} \binom{N-1}{k-1} \frac{1}{r^{N-1}} [r^{N-k}(u')^k]' + \lambda f(r),
\end{equation}
where the radial Laplacian is given by $\Delta_r(\cdot)=\frac{1}{r^{N-1}} [r^{N-1} (\cdot)']'$
and the radial coordinate $r \in [0,1]$.
Now integrating with respect to $r$, applying the boundary condition $u'(0)=0$
(the other boundary conditions for the Dirichlet problem are $u(1)=u'(1)=0$),
and substituting $v=u'$ we arrive to
\begin{equation}\nonumber
v'' + \frac{N-1}{r} v' - \frac{N-1}{r^2} v = \frac{(-1)^k}{k} \binom{N-1}{k-1} \frac{v^k}{r^{k-1}}
+ \frac{\lambda}{r^{N-1}} \int_0^{r} f(s) s^{N-1} ds,
\end{equation}
subject to the boundary conditions $v(0)=v(1)=0$.
The change of variables $w(t)=-v(e^{-t})$ leads to the boundary value problem
\begin{equation}
\nonumber
\left\{ \begin{array}{rcl}
-w'' + (N-2)w' + (N-1)w &=& \frac{1}{k} \binom{N-1}{k-1} e^{(k-3)t} w^k \\
&+& \lambda e^{(N-3)t} \int_0^{e^{-t}} f(s) s^{N-1} ds, \\
w(0)=w(+\infty) &=& 0,
\end{array}\right.
\end{equation}
where $k, N \in \mathbb{N}$, $2 \le k \le N$, and $t \in [0,+\infty[$.
Finally we restate this problem as
\begin{equation}
\label{dirichletr}
\left\{ \begin{array}{rcl}
-w'' + (N-2)w' + (N-1)w &=& \frac{1}{k} \binom{N-1}{k-1} e^{(k-3)t} w^k \\
&+& \lambda e^{(N-3)t} \int_0^{e^{-t}} g(s) \, ds, \\
w(0)=w(+\infty) &=& 0,
\end{array}\right.
\end{equation}
where $g(\cdot) \in L^1([0,1])$, which is the form we are going to analyze.

The problem corresponding to Navier boundary conditions in the radial setting reads
\begin{equation}
\label{navierr}
\left\{ \begin{array}{rcl}
-w'' + (N-2)w' + (N-1)w &=& \frac{1}{k} \binom{N-1}{k-1} e^{(k-3)t} w^k \\
&+& \lambda e^{(N-3)t} \int_0^{e^{-t}} g(s) \, ds, \\
w'(0) - (N-1) w(0) = w(+\infty) &=& 0,
\end{array}\right.
\end{equation}
after the same changes of variables have been carried out.

All along this work we will focus on low dimensions $N=2,3$ and on quadratic nonlinearities $k=2$. We do so because $N=4$ is a critical dimension
and $k=3$ is a critical nonlinearity for this model, in a sense that we will not make precise herein. We will leave the investigation of these
cases for the future. Note also that we will consider the higher order symmetry condition $u'''(0)=0$, which will constitute the fourth
boundary condition, at all times. This condition is necessary in order to build an existence theory that is compatible with the
Carath\'eodory notion of solution.

These are the main results proven in this paper:

\begin{theorem}\label{main1}
The radial equation~\eqref{radialhessian} with $k=2$, $N=2,3$ and
subject to either homogeneous Dirichlet or Navier boundary conditions possesses at least one solution
in $C^3([0,1]) \cap W^{4,1}\left([0,1],r^{N-1}dr \right)$ provided $|\lambda|$ is small enough.
Moreover, there exists an open ball in this functional space that contains the origin and strictly one solution.
\end{theorem}

\begin{proof}
The same result with the solutions in $C^3([0,1]) \cap AC^3([\varepsilon,1]) \, \forall \, \varepsilon >0$ is a direct consequence of the changes of variables
above (note that the change of variables that involves the independent variable is smooth and monotonic wherever it is defined)
and the auxiliary results in the text that lead to Theorem~\ref{regu},
Remark~\ref{regure}, Theorem~\ref{regu2} and Remark~\ref{regure2}.

To finish the proof we need to consider the higher order symmetry condition $u'''(0)=0$. Note that it is clear that
$$
\int_\varepsilon^1 |u^{(\i v)}| \, r^{N-1} \,dr < \infty \,\, \forall \,\, \varepsilon >0.
$$
Therefore we just need to estimate
$$
\int_0^\varepsilon |u^{(\i v)}| \, r^{N-1} \, dr \le \int_0^\varepsilon \left( \frac{|u'''|}{r} + C \right) r^{N-1} \, dr < \infty,
$$
for $\varepsilon$ small enough, where $C$ is a positive constant and
we have used the higher order symmetry condition in the first inequality and the fact that $u \in C^3([0,1])$ in the second.
\end{proof}

\begin{remark}
We cannot conclude that these two solutions belong to the functional space $AC^3([0,1])$ as the singularity in the origin
prevents us to get this higher regularity (note indeed that $AC^3([0,1]) \subset C^3([0,1]) \cap W^{4,1}\left([0,1],r^{N-1}dr \right)$).
However, if we looked for radial solutions in a ring rather than a ball, then Theorem~\ref{main1} would still be valid but this time
the two solutions would indeed belonged to $AC^3([a,b])$, where $a$ and $b$ are the inner and outer radiuses of the ring respectively.
The proofs would follow analogously to the corresponding ones in this work.
Note that in both cases the notion of solution corresponds to the Carath\'eodory one.
\end{remark}

\begin{theorem}\label{main2}
The radial equation~\eqref{radialhessian} with $k=2$, $N=2,3$ and
subject to either homogeneous Dirichlet or Navier boundary conditions possesses at least one solution
in $C^3([0,1]) \cap W^{4,1}\left([0,1],r^{N-1}dr \right)$ provided $\lambda$ is small enough.
\end{theorem}

\begin{proof}
This result is a consequence of Theorems~\ref{th_exist} and~\ref{th_exist2}, as well as the proof of Theorem~\ref{main1}.
\end{proof}

\begin{remark}
Note this result improves the previous one in the sense that $\lambda$ can be of arbitrarily large magnitude provided it is negative,
at the price of loosing the non-degeneracy property near the origin.
\end{remark}

\begin{theorem}\label{main3}
The radial equation~\eqref{radialhessian} with $k=2$, $N=2,3$ and
subject to either homogeneous Dirichlet or Navier boundary conditions possesses no weak solutions
provided $\lambda$ is large enough.
\end{theorem}

\begin{proof}
The statement follows from Theorems~\ref{nondir} and~\ref{nonnav} in the text.
\end{proof}

\begin{remark}
By weak solution we mean a function $w$ that is square integrable together with its first derivative and obeys the equation in a suitable
weak sense. For the precise definitions see the next section. Of course, non-existence of weak solutions implies non-existence of strong solutions
in the sense of Theorems~\ref{main1} and~\ref{main2}.
\end{remark}

\section{Existence results for $k=2$ and $N \in \{2,3\}$}
\label{exisrel}

From now on we will employ the notation $\mathbb{R}_+ := [0,\infty[$, $\| \cdot \|_p := \| \cdot \|_{L^p(\mathbb{R}_+)}$ and
$\| \cdot \|_\infty := \| \cdot \|_{L^\infty(\mathbb{R}_+)}$.

We begin with some preliminary results.

\subsection{Dirichlet problem}
\label{sec:exdir}

In this subsection we focus on the problem
\begin{equation}
\label{dirichlet2}
\left\{ \begin{array}{rcl}
-w'' + (N-2)w' + (N-1)w &=& \frac{N-1}{2} e^{-t} w^2 \\
&+& \lambda e^{(N-3)t} \int_0^{e^{-t}} g(s) \, ds, \\
w(0)=w(\infty) &=& 0.
\end{array}\right.
\end{equation}

\begin{lemma}\label{datum}
The function
$$
h_1(t)=e^{(N-3)t} \int_0^{e^{-t}} g(s) \, ds,
$$
where $g \in L^1(0,1)$, belongs to $BC(\mathbb{R}_+) \cap L^1(\mathbb{R}_+)$ if $N=2$ and to $BC(\mathbb{R}_+)$ if $N=3$. Moreover
$$
\lim_{t \to \infty} h_1(t) =0
$$
in either case.
\end{lemma}

\begin{proof}
Since $g$ is Lebesgue integrable, then it is obvious that $h_1(t)$ is continuous, as it is the product of a continuous function and
the composition of two continuous functions. For $N=2,3$ we have
$$
\left| e^{(N-3)t} \int_0^{e^{-t}} g(s) \, ds \right| \le \| g \|_{L^1(0,1)}.
$$
Since this bound is uniform in $t$ we can take the $\sup_{0 \le t < \infty}$ on its left hand side to conclude $h_1(t) \in BC(\mathbb{R}_+)$.

If $N=2$ then
$$
h_1(t)= e^{-t} \int_0^{e^{-t}} g(s) \, ds.
$$
Therefore
\begin{eqnarray}\nonumber
\| h_1(t) \|_{L^1(\mathbb{R}_+)} = \int_0^\infty \left| e^{-t} \int_0^{e^{-t}} g(s) \, ds \right| dt
&\le& \| g \|_{L^1(0,1)} \int_0^\infty e^{-t} \, dt \\ \nonumber
&=& \| g \|_{L^1(0,1)}.
\end{eqnarray}

Finally, since $N=2,3$ and $g \in L^1(0,1)$ we have
$$
|h_1(t)| \le \int_0^{e^{-t}} |g(s)| \, ds \longrightarrow 0, \quad \mathrm{as} \quad t \to \infty.
$$
\end{proof}

\begin{remark}\label{acdat}
Actually $h_1(t)$ is slightly more regular, as $h_1 \in AC(\mathbb{R}_+)$ for $N=3$
and $h_1 \in AC(\mathbb{R}_+) \cap L^1(\mathbb{R}_+)$ for $N=2$.
This follows from the fact that
$h_1(t)$ can be regarded as the convolution of an absolutely continuous function with a function ($e^{-t}$) that is both absolutely continuous
and monotonic. We actually need this improved regularity in the proof of Theorem~\ref{main1}.
\end{remark}

From now onwards we will denote $h_\lambda = \lambda h_1$.
As we will see in the following it is natural to work with $\mu-$integrable functions,
where $\mu$ is the absolutely continuous measure uniquely defined by its
Radon-Nikodym derivative $d\mu = e^{(2-N)t} \, dt$.

\begin{definition}
Let $v: \mathbb{R}_+ \longrightarrow \mathbb{R}$.
We define the norm
$$
\| v \|_\mu := \left\{ \int_0^\infty [(v')^2 + v^2] e^{(2-N)t} \, dt \right\}^{1/2}.
$$
We define the space $H^1_\mu$ as the completion of $C_0^\infty(\mathbb{R}_+)$ with respect to the norm $\| \cdot \|_\mu$.
\end{definition}

\begin{lemma}\label{linapp}
The singular boundary value problem
\begin{eqnarray} \nonumber
-w'' + (N-2)w' + (N-1)w = h_\lambda(t) + h^*(t), \\ \nonumber
w(0)=w(\infty) = 0, \\ \nonumber
h^*(t) \in L^1(\mathbb{R}_+), \\ \nonumber
\end{eqnarray}
has a unique solution $w \in H^1_\mu$.
\end{lemma}

\begin{proof}
The unique solution is explicit and can be calculated by means of variation of parameters,
\begin{eqnarray}\nonumber
w(t) &=& -\frac{e^{-t}}{N} \int_0^\infty e^{(1-N)s} \, h(s) \, ds + \frac{e^{-t}}{N} \int_0^t e^s \, h(s) \, ds \\ \nonumber
&+& \frac{e^{(N-1)t}}{N} \int_t^\infty e^{(1-N)s} \, h(s) \, ds,
\end{eqnarray}
where $h(t):=h^*(t)+h_\lambda(t)$.
One can check that the formula above is well defined and fulfils the correct boundary conditions
given the regularity of $h_\lambda$ proven in Lemma~\ref{datum}. It remains to prove that
$w \in  H^1_\mu$. This is done separately for the cases $N=2$ and $N=3$.

\textsc{Step 1: $N=2$.} In this case $h \in L^1(\mathbb{R}_+)$. We can multiply our equation
$$
-w'' + w = h(t)
$$
by $w$ and integrate over $[0,\infty[$ to get
$$
\int_0^\infty (w')^2 \, dt + \int_0^\infty w^2 \, dt = \int_0^\infty h \, w \, dt,
$$
after integrating by parts and applying the boundary conditions. Now, by means of a H\"older inequality
and a Sobolev embedding we find
$$
\left| \int_0^\infty h \, w \, dt \right| \le \| h \|_{1} \| w  \|_\infty \le C \| h \|_{1} \| w \|_{H^1(\mathbb{R}_+)}.
$$
The fact that
$$
\| w \|_\mu^2 = \int_0^\infty [(v')^2 + v^2] \, dt = \| w \|_{H^1(\mathbb{R}_+)}^2
$$
in $N=2$ allows us to conclude
$$
\| w \|_\mu \le C \| h \|_{1}.
$$

\textsc{Step 2: $N=3$.} Now the equation reads
$$
-w'' + w' + 2w = h_\lambda(t) + h^*(t).
$$
Multiplying it by $e^{-t} w$ and integrating over $[0,\infty[$ we find
$$
\int_0^\infty e^{-t} (w')^2 \, dt + 2 \int_0^\infty e^{-t} w^2 \, dt = \int_0^\infty h_\lambda \, e^{-t} w \, dt + \int_0^\infty h^* \, e^{-t} w \, dt,
$$
after integrating by parts and applying the boundary conditions. Now we can estimate
\begin{eqnarray}\nonumber
\left| \int_0^\infty h_\lambda \, e^{-t} w \, dt \right| &\le& \| h_\lambda \|_\infty \left( \int_0^\infty |w|^2 e^{-t} \, dt \right)^{1/2}
\left( \int_0^\infty e^{-t} \, dt \right)^{1/2} \\ \nonumber
&=& \| h_\lambda \|_\infty \left( \int_0^\infty |w|^2 e^{-t} \,dt \right)^{1/2},
\end{eqnarray}
which result from the application of H\"older inequality. And for the term containing $h^*$ we have
\begin{eqnarray}\nonumber
\left| \int_0^\infty h^* \, e^{-t} w \, dt \right| &\le& \| h^* \|_1 \| e^{-t} w \|_\infty
\le C \| h^* \|_1 \| e^{-t} w \|_{H^1(\mathbb{R}_+)} \\ \nonumber
&=& C \| h^* \|_1 \left[ \int_0^\infty |w'|^2 e^{-2t} \, dt \right]^{1/2} \le C \| h^* \|_1 \| w \|_\mu,
\end{eqnarray}
where we have employed, in this order, H\"older inequality, a Sobolev embedding, the equality
$$
\int_0^\infty [|(e^{-t} w)'|^2 + (e^{-t} w)^2] \, dt = \int_0^\infty |w'|^2 e^{-2t} \, dt,
$$
which results from integration by parts and the application of the boundary conditions,
and again H\"older inequality
$$
\int_0^\infty |w'|^2 e^{-2t} \, dt \le \int_0^\infty |w'|^2 e^{-t} \, dt.
$$
Rearranging both estimates we can conclude
$$
\| w \|_\mu \le C ( \| h^* \|_1 + \| h_\lambda \|_\infty ).
$$
\end{proof}

Now, we are able to prove the main result of this subsection.

\begin{theorem}\label{existdir}
The singular boundary value problem~\eqref{dirichletr} possesses at least one solution $w \in H^1_\mu$ for $k=2$ and $N \in \{2,3\}$
provided $|\lambda|$ is small enough.
\end{theorem}

\begin{proof}
Fix $\phi_1, \phi_2 \in H^1_\mu$. Then it is clear, by Lemma~\ref{linapp}, that the linear boundary value problem
$$
\left\{ \begin{array}{rcl}
-w_i'' + (N-2)w_i' + (N-1)w_i &=& \frac{N-1}{2} e^{-t} \phi_i^2 \\
&+& \lambda e^{(N-3)t} \int_0^{e^{-t}} g(s) \, ds, \\
w_i(0)=w_i(\infty) &=& 0.
\end{array}\right.
$$
has a unique solution $w_i \in H^1_\mu$ for $i=1,2$. Subtracting both problems, for $w_3 := w_1 - w_2$, one finds
$$
\left\{ \begin{array}{rcl}
-w_3'' + (N-2)w_3' + (N-1)w_3 &=& \frac{N-1}{2} e^{-t} (\phi_1^2 - \phi_2^2) \\
w_3(0)=w_3(\infty) &=& 0.
\end{array}\right.
$$

\textsc{Step 1: $N=2$.} In this case $\phi_1, \phi_2, w_3 \in H^1(\mathbb{R}_+)$ and
$$
-w_3'' + w_3 = \frac{1}{2} e^{-t} (\phi_1^2 - \phi_2^2).
$$
Multiplying this equation by $w_3$ and integrating over $[0,\infty[$, and then integrating by parts and applying the boundary conditions yields
\begin{eqnarray}\nonumber
\| w_3 \|_{\mu}^2 &=&
\int_0^\infty (w_3')^2 \, dt + \int_0^\infty w_3^2 \, dt = \frac{1}{2} \int_0^\infty e^{-t} (\phi_1^2 - \phi_2^2) \, w_3 \, dt \\ \nonumber
&\le& \int_0^\infty |\phi_1 - \phi_2| \, |\phi_1 + \phi_2| \, |w_3| \, dt \\ \nonumber
&\le& \left(\| \phi_1 \|_\infty + \| \phi_2 \|_\infty \right) \left( \int_0^\infty w_3^2 \, dt \right)^{1/2}
\left( \int_0^\infty |\phi_1 - \phi_2|^2 \, dt \right)^{1/2} \\ \nonumber
&\le& C \left(\| \phi_1 \|_\mu + \| \phi_2 \|_\mu \right) \| w_3 \|_\mu \| \phi_1 -\phi_2 \|_\mu,
\end{eqnarray}
where the inequalities come from two different H\"older inequalities in the first and second cases (combined with a triangle inequality in the second),
and a Sobolev embedding in the third.
So we conclude
$$
\| w_3 \|_{\mu} \le C \left(\| \phi_1 \|_\mu + \| \phi_2 \|_\mu \right) \| \phi_1 -\phi_2 \|_\mu.
$$

\textsc{Step 2: $N=3$.} In this case the equation for $w_3$ reads
$$
-w''_3 + w_3' + 2 w_3 = e^{-t} (\phi_1^2 - \phi_2^2).
$$
Now we multiply this equation by $e^{-t} w_3$ and integrate over $[0,\infty[$; after integrating by parts and applying the boundary conditions
we get
\begin{eqnarray}\nonumber
\| w_3 \|_{\mu}^2 &\le& \int_0^\infty (w_3')^2 \, e^{-t} \, dt + 2 \int_0^\infty w_3^2 \, e^{-t} \, dt =
\int_0^\infty e^{-2t} (\phi_1^2 - \phi_2^2) \, w_3 \, dt \\ \nonumber
&\le& \int_0^\infty |\phi_1 - \phi_2| \, e^{-t/2} \, |\phi_1 + \phi_2| \, e^{-t} \, |w_3| \, e^{-t/2} \, dt \\ \nonumber
&\le& \left(\| \phi_1 \, e^{-t} \|_\infty + \| \phi_2 \, e^{-t} \|_\infty \right) \left( \int_0^\infty w_3^2 \, e^{-t} \, dt \right)^{1/2}
\\ \nonumber & & \times \left( \int_0^\infty |\phi_1 - \phi_2|^2 \, e^{-t} \, dt \right)^{1/2} \\ \nonumber
&\le& C \left(\| \phi_1 \|_\mu + \| \phi_2 \|_\mu \right) \| w_3 \|_\mu \| \phi_1 -\phi_2 \|_\mu,
\end{eqnarray}
where the inequalities come from introducing the absolute value inside the integrand in the first case,
a H\"older inequality combined with a triangle inequality in the second one,
and finally, in the third, the Sobolev embedding
$$
\| \phi_i \, e^{-t} \|_\infty \le C \| \phi_i \, e^{-t} \|_{H^1(\mathbb{R}_+)}, \quad \text{for} \quad i=1,2,
$$
combined with
$$
\| \phi_i \, e^{-t} \|_{H^1(\mathbb{R}_+)} = \left( \int_0^\infty (\phi_i')^2 \, e^{-2t} \, dt \right)^{1/2} \le
\left( \int_0^\infty (\phi_i')^2 \, e^{-t} \, dt \right)^{1/2},
$$
for $i=1,2$, where the equality comes from integration by parts and the application of the boundary conditions,
and the inequality is H\"older inequality.
So we conclude again
$$
\| w_3 \|_{\mu} \le C \left(\| \phi_1 \|_\mu + \| \phi_2 \|_\mu \right) \| \phi_1 -\phi_2 \|_\mu.
$$

\textsc{Step 3: Banach Fixed Point Theorem.} Following Step $1$ and $2$ we know that
\begin{equation}\label{ineqban}
\| w_1 - w_2 \|_{\mu} \le C \left(\| \phi_1 \|_\mu + \| \phi_2 \|_\mu \right) \| \phi_1 -\phi_2 \|_\mu,
\end{equation}
for both $N=2$ and $N=3$. Now consider
$$
\left\{ \begin{array}{rcl}
-w_0'' + (N-2)w_0' + (N-1)w_0 &=& \lambda e^{(N-3)t} \int_0^{e^{-t}} g(s) \, ds, \\
w_0(0)=w_0(\infty) &=& 0.
\end{array}\right.
$$
It is clear that the unique solution to this problem fulfils
\begin{eqnarray}\nonumber
\|w_0\|_\mu \le C \, |\lambda| \, \| h_\lambda \|_1 \quad &\text{for}& \quad N=2, \\ \nonumber
\|w_0\|_\mu \le C \, |\lambda| \, \| h_\lambda \|_\infty \quad &\text{for}& \quad N=3,
\end{eqnarray}
which is a direct consequence of Lemma~\ref{linapp}. Now we choose $\phi_1, \phi_2 \in \mathcal{B}$,
where
$$
\mathcal{B}:= \{ \varphi \in H^1_\mu \, | \, \| \varphi - w_0 \| \le \rho \},
$$
is the ball of center $w_0$ and radius $\rho$ in the Banach space $H^1_\mu$.
The triangle inequality allows us to translate~\eqref{ineqban} into
\begin{eqnarray}\nonumber
\| w_1 - w_2 \|_{\mu} &\le& C \left(\rho + \| w_0 \|_\mu \right) \| \phi_1 -\phi_2 \|_\mu \\ \label{contract} &\le&
\left\{ \begin{array}{rcl}
C \left(\rho + |\lambda| \, \| h_\lambda \|_1 \right) \| \phi_1 -\phi_2 \|_\mu \quad &\text{for}& \quad N=2 \\
C \left(\rho + |\lambda| \, \| h_\lambda \|_\infty \right) \| \phi_1 -\phi_2 \|_\mu \quad &\text{for}& \quad N=3
\end{array}\right. .
\end{eqnarray}
Note also
$$
\| w_i - w_0 \|_{\mu} \le C \| \phi_i \|_\mu^2 \quad \text{for} \quad i=1,2.
$$
The triangle inequality leads again to
\begin{eqnarray}\nonumber
\| w_i - w_0 \|_{\mu} &\le& C \| \phi_i \|_\mu^2 \le C ( \rho + \| w_0 \|_\mu )^2 \\ \label{ball2ball}
&\le& \left\{ \begin{array}{rcl}
C \left(\rho^2 + \lambda^2 \| h_\lambda \|_1^2 \right) \quad &\text{for}& \quad N=2 \\
C \left(\rho^2 + \lambda^2 \| h_\lambda \|_\infty^2 \right) \quad &\text{for}& \quad N=3
\end{array}\right. .
\end{eqnarray}
For any $\varphi \in \mathcal{B}$, we define the nonlinear operator
\begin{eqnarray}\nonumber
\mathcal{T}: \mathcal{B} &\longrightarrow& \mathcal{B} \\ \nonumber
\varphi &\longmapsto& \mathcal{T}(\varphi)=w,
\end{eqnarray}
where $w$ is the unique solution to
$$
\left\{ \begin{array}{rcl}
-w'' + (N-2)w' + (N-1)w &=& \frac{N-1}{2} e^{-t} \varphi^2 \\
&+& \lambda e^{(N-3)t} \int_0^{e^{-t}} g(s) \, ds, \\
w(0)=w(\infty) &=& 0.
\end{array}\right.
$$
The operator $\mathcal{T}$ so defined is well-defined (in the sense that it is bounded in the ball $\mathcal{B}$)
by~\eqref{ball2ball} and contractive by~\eqref{contract}
provided $\rho$ and $|\lambda|$ are small enough. Consequently, by the Banach fixed point theorem, there is a unique fixed point
of $\mathcal{T}$, which in turn solves~\eqref{dirichletr}.
\end{proof}

\begin{remark}
The ball $\mathcal{B}$ contains the origin of $H^1_\mu$ provided $|\lambda|$ is small enough.
\end{remark}


\subsection{Navier problem}
\label{sec:exnav}

In this subsection we concentrate on the problem
\begin{equation}
\label{navier}
\left\{ \begin{array}{rcl}
-w'' + (N-2)w' + (N-1)w &=& \frac{N-1}{2} e^{-t} w^2 \\
&+& \lambda e^{(N-3)t} \int_0^{e^{-t}} g(s) \, ds, \\
w'(0) - (N-1) w(0) = w(+\infty) &=& 0.
\end{array}\right.
\end{equation}

We will adapt the arguments in the previous section to the present setting.
First we need to slightly modify our functional framework.

\begin{definition}
We define the space $\tilde{H}^1_\mu$ as the set of all measurable functions
$v:\mathbb{R}_+ \longrightarrow \mathbb{R}$ with a finite norm $\| \cdot \|_\mu$.
\end{definition}

\begin{remark}
In principle it is not evident how it is possible to look for weak solutions of problem~\eqref{navier} in the functional space $\tilde{H}^1_\mu$.
Unlike problem~\eqref{dirichlet2}, which solutions are continuous on bounded intervals as can be deduced by Sobolev embedding,
in the present case we should look for
a different interpretation of the boundary condition involving one derivative evaluated at the origin.
Note that for $w \in \tilde{H}^1_\mu$ the nonlinearity in problem~\eqref{navier} is summable and therefore we can turn this differential equation into
an integral one via the use of the variation of cons\-tants formula below in~\eqref{varcon}. We can give in this way a rigourous meaning to the quantity $w'(0)$.
\end{remark}

\begin{lemma}\label{linapp2}
The singular boundary value problem
\begin{eqnarray} \nonumber
-w'' + (N-2)w' + (N-1)w = h_\lambda(t) + h^*(t), \\ \nonumber
w'(0) - (N-1) w(0) = w(\infty) = 0, \\ \nonumber
h^*(t) \in L^1(\mathbb{R}_+), \\ \nonumber
\end{eqnarray}
has a unique solution $w \in \Tilde{H}^1_\mu$.
\end{lemma}

\begin{proof}
The unique solution is explicitly given by the variation of constants formula
\begin{equation}\label{varcon}
w(t) = \frac{e^{-t}}{N} \int_0^t e^s \, h(s) \, ds + \frac{e^{(N-1)t}}{N} \int_t^\infty e^{(1-N)s} \, h(s) \, ds,
\end{equation}
where $h(t):=h^*(t)+h_\lambda(t)$.
It is easy to check that this solution is well defined and fulfils the correct boundary conditions.

\textsc{Step 1: $N=2$.} We have $h \in L^1(\mathbb{R}_+)$. Testing the equation against $w$ we find
$$
w(0)^2 + \int_0^\infty (w')^2 \, dt + \int_0^\infty w^2 \, dt = \int_0^\infty h \, w \, dt,
$$
as a result of integration by parts and the application of the boundary conditions. The H\"older inequality
and Sobolev embedding
$$
\left| \int_0^\infty h \, w \, dt \right| \le \| h \|_{1} \| w  \|_\infty \le C \| h \|_{1} \| w \|_{H^1(\mathbb{R}_+)}.
$$
lead us to
$$
\| w \|_\mu \le C \| h \|_{1}.
$$

\textsc{Step 2: $N=3$.} In this case we test the equation against
$e^{-t} w$ to get
\begin{eqnarray}\nonumber
& & 2 w(0)^2 + \int_0^\infty e^{-t} (w')^2 \, dt + 2 \int_0^\infty e^{-t} w^2 \, dt \\ \nonumber
&=& \int_0^\infty h_\lambda \, e^{-t} w \, dt + \int_0^\infty h^* \, e^{-t} w \, dt,
\end{eqnarray}
after integrating by parts and applying the boundary conditions. For the first summand on the right hand side, one has
\begin{eqnarray}\nonumber
\left| \int_0^\infty h_\lambda \, e^{-t} w \, dt \right| &\le& \| h_\lambda \|_\infty \left( \int_0^\infty |w|^2 e^{-t} \, dt \right)^{1/2}
\left( \int_0^\infty e^{-t} \, dt \right)^{1/2} \\ \nonumber
&=& \| h_\lambda \|_\infty \left( \int_0^\infty |w|^2 e^{-t} \,dt \right)^{1/2},
\end{eqnarray}
as a consequence of H\"older inequality. For the second summand, we can establish the estimate
\begin{eqnarray}\nonumber
\left| \int_0^\infty h^* \, e^{-t} w \, dt \right| &\le& \| h^* \|_1 \| e^{-t} w \|_\infty
\le C \| h^* \|_1 \| e^{-t} w \|_{H^1(\mathbb{R}_+)} \\ \nonumber
&=& C \| h^* \|_1 \left[ \int_0^\infty |w'|^2 e^{-2t} \, dt \right]^{1/2} + C \| h^* \|_1 \left| w(0) \right| \\ \nonumber
&\le& C \| h^* \|_1 \| w \|_\mu + C \| h^* \|_1 \left| w(0) \right|,
\end{eqnarray}
which follows from the application of H\"older inequality, a Sobolev embedding, the equality
$$
\int_0^\infty [|(e^{-t} w)'|^2 + (e^{-t} w)^2] \, dt = w(0)^2 + \int_0^\infty |w'|^2 e^{-2t} \, dt,
$$
which is a consequence of integration by parts and the application of the boundary conditions,
together with the sublinearity of the square root,
and a new H\"older inequality
$$
\int_0^\infty |w'|^2 e^{-2t} \, dt \le \int_0^\infty |w'|^2 e^{-t} \, dt.
$$

We just need to estimate
$$
w(0)=-\int_0^\infty \left[w(t) e^{-t}\right]' dt = \int_0^\infty \left[ w(t) e^{-t} - w'(t) e^{-t} \right] dt,
$$
and thus
\begin{eqnarray}\nonumber
|w(0)| &\le& \int_0^\infty \left[ |w(t)| + |w'(t)| \right] e^{-t} dt \\ \nonumber
&\le& \left\{ \int_0^\infty \left[ |w(t)| + |w'(t)| \right]^2 e^{-t} dt \right\}^{1/2} \\ \nonumber
&\le& \sqrt{2} \left\{ \int_0^\infty \left[ |w(t)|^2 + |w'(t)|^2 \right] e^{-t} dt \right\}^{1/2},
\end{eqnarray}
where we have used, in this order, the triangle, H\"{o}lder and Young inequalities.

Rearranging all estimates allows us to conclude
$$
\| w \|_\mu \le C ( \| h^* \|_1 + \| h_\lambda \|_\infty ).
$$
\end{proof}

Now we proceed to the main result of this subsection.

\begin{theorem}\label{existnavier}
The singular boundary value problem~\eqref{navier} possesses at least one solution $w \in \Tilde{H}^1_\mu$ for $k=2$ and $N \in \{2,3\}$
provided $|\lambda|$ is small enough.
\end{theorem}

\begin{proof}
Fix $\phi_1, \phi_2 \in \Tilde{H}^1_\mu$. Then $e^{-t} \phi_i^2 \in L^1(\mathbb{R}_+)$ and by Lemma~\ref{linapp2} there is a unique solution to
$$
\left\{ \begin{array}{rcl}
-w_i'' + (N-2)w_i' + (N-1)w_i &=& \frac{N-1}{2} e^{-t} \phi_i^2 \\
&+& \lambda e^{(N-3)t} \int_0^{e^{-t}} g(s) \, ds, \\
w_i(0)=w_i(\infty) &=& 0.
\end{array}\right.
$$
such that $w_i \in \Tilde{H}^1_\mu$ for $i=1,2$. Subtracting both problems and denoting $w_3 := w_1 - w_2$ we find
$$
\left\{ \begin{array}{rcl}
-w_3'' + (N-2)w_3' + (N-1)w_3 &=& \frac{N-1}{2} e^{-t} (\phi_1^2 - \phi_2^2) \\
w_3(0)=w_3(\infty) &=& 0.
\end{array}\right.
$$

\textsc{Step 1: $N=2$.} Arguing as in the proof of Lemma~\ref{linapp2} we get
\begin{eqnarray}\nonumber
\| w_3 \|_{\mu}^2 &\le& w_3(0)^2 +
\int_0^\infty (w_3')^2 \, dt + \int_0^\infty w_3^2 \, dt \\ \nonumber
&=& \frac{1}{2} \int_0^\infty e^{-t} (\phi_1^2 - \phi_2^2) \, w_3 \, dt \\ \nonumber
&\le& \int_0^\infty |\phi_1 - \phi_2| \, |\phi_1 + \phi_2| \, |w_3| \, dt \\ \nonumber
&\le& \left(\| \phi_1 \|_\infty + \| \phi_2 \|_\infty \right) \left( \int_0^\infty w_3^2 \, dt \right)^{1/2}
\left( \int_0^\infty |\phi_1 - \phi_2|^2 \, dt \right)^{1/2} \\ \nonumber
&\le& C \left(\| \phi_1 \|_\mu + \| \phi_2 \|_\mu \right) \| w_3 \|_\mu \| \phi_1 -\phi_2 \|_\mu,
\end{eqnarray}
where the last three inequalities are, respectively,
H\"older inequality, another H\"older inequality acting together with the triangle inequality,
and a final Sobolev embedding.
Summing up
$$
\| w_3 \|_{\mu} \le C \left(\| \phi_1 \|_\mu + \| \phi_2 \|_\mu \right) \| \phi_1 -\phi_2 \|_\mu.
$$

\textsc{Step 2: $N=3$.} Following again the proof of Lemma~\ref{linapp2} we find
\begin{eqnarray}\nonumber
\| w_3 \|_{\mu}^2 &\le& 2 w(0)^2 + \int_0^\infty (w_3')^2 \, e^{-t} \, dt + 2 \int_0^\infty w_3^2 \, e^{-t} \, dt \\ \nonumber &=&
\int_0^\infty e^{-2t} (\phi_1^2 - \phi_2^2) \, w_3 \, dt \\ \nonumber
&\le& \int_0^\infty |\phi_1 - \phi_2| \, e^{-t/2} \, |\phi_1 + \phi_2| \, e^{-t} \, |w_3| \, e^{-t/2} \, dt \\ \nonumber
&\le& \left(\| \phi_1 \, e^{-t} \|_\infty + \| \phi_2 \, e^{-t} \|_\infty \right) \left( \int_0^\infty w_3^2 \, e^{-t} \, dt \right)^{1/2}
\\ \nonumber & & \times \left( \int_0^\infty |\phi_1 - \phi_2|^2 \, e^{-t} \, dt \right)^{1/2} \\ \nonumber
&\le& C \left(\| \phi_1 \|_\mu + \| \phi_2 \|_\mu \right) \| w_3 \|_\mu \| \phi_1 -\phi_2 \|_\mu,
\end{eqnarray}
where the last two inequalities are a H\"older inequality applied simultaneously with a triangle inequality,
and then the Sobolev embedding
$$
\| \phi_i \, e^{-t} \|_\infty \le C \| \phi_i \, e^{-t} \|_{H^1(\mathbb{R}_+)}, \quad \text{for} \quad i=1,2,
$$
combined with
\begin{eqnarray}\nonumber
\| \phi_i \, e^{-t} \|_{H^1(\mathbb{R}_+)} &\le& |\phi_i(0)| + \left( \int_0^\infty (\phi_i')^2 \, e^{-2t} \, dt \right)^{1/2} \\ \nonumber
&\le& |\phi_i(0)| + \left( \int_0^\infty (\phi_i')^2 \, e^{-t} \, dt \right)^{1/2} \\ \nonumber
&\le& \left( 1 + \sqrt{2} \right) \| \phi \|_\mu,
\end{eqnarray}
for $i=1,2$, which is a calculation analogous to that in the proof of Lemma~\ref{linapp2}.

These estimates lead us again to
$$
\| w_3 \|_{\mu} \le C \left(\| \phi_1 \|_\mu + \| \phi_2 \|_\mu \right) \| \phi_1 -\phi_2 \|_\mu.
$$

\textsc{Step 3: Banach Fixed Point Theorem.} We have already obtained
\begin{equation}\label{ineqban2}
\| w_1 - w_2 \|_{\mu} \le C \left(\| \phi_1 \|_\mu + \| \phi_2 \|_\mu \right) \| \phi_1 -\phi_2 \|_\mu,
\end{equation}
for $N=2,3$. The linear problem
$$
\left\{ \begin{array}{rcl}
-w_0'' + (N-2)w_0' + (N-1)w_0 &=& \lambda e^{(N-3)t} \int_0^{e^{-t}} g(s) \, ds, \\
w_0'(0) - (N-1) w_0(0)=w_0(\infty) &=& 0.
\end{array}\right.
$$
has a unique solution in $\Tilde{H}^1_\mu$ such that
\begin{eqnarray}\nonumber
\|w_0\|_\mu \le C \, |\lambda| \, \| h_\lambda \|_1 \quad &\text{in}& \quad N=2, \\ \nonumber
\|w_0\|_\mu \le C \, |\lambda| \, \| h_\lambda \|_\infty \quad &\text{in}& \quad N=3,
\end{eqnarray}
as can be deduced from Lemma~\ref{linapp2}. We select $\phi_1, \phi_2 \in \mathcal{B}$ for
$$
\mathcal{B}:= \{ \varphi \in \Tilde{H}^1_\mu \, | \, \| \varphi - w_0 \| \le \rho \},
$$
being the ball of center $w_0$ and radius $\rho$ in the Banach space $\Tilde{H}^1_\mu$.
By the triangle inequality we may transform~\eqref{ineqban2} into
\begin{eqnarray}\nonumber
\| w_1 - w_2 \|_{\mu} &\le& C \left(\rho + \| w_0 \|_\mu \right) \| \phi_1 -\phi_2 \|_\mu \\ \label{contract2} &\le&
\left\{ \begin{array}{rcl}
C \left(\rho + |\lambda| \, \| h_\lambda \|_1 \right) \| \phi_1 -\phi_2 \|_\mu \quad &\text{in}& \quad N=2 \\
C \left(\rho + |\lambda| \, \| h_\lambda \|_\infty \right) \| \phi_1 -\phi_2 \|_\mu \quad &\text{in}& \quad N=3
\end{array}\right. ,
\end{eqnarray}
and this particular version of~\eqref{ineqban2}
$$
\| w_i - w_0 \|_{\mu} \le C \| \phi_i \|_\mu^2 \quad \text{for} \quad i=1,2,
$$
into
\begin{eqnarray}\nonumber
\| w_i - w_0 \|_{\mu} &\le& C \| \phi_i \|_\mu^2 \le C ( \rho + \| w_0 \|_\mu )^2 \\ \label{ball2ball2}
&\le& \left\{ \begin{array}{rcl}
C \left(\rho^2 + \lambda^2 \| h_\lambda \|_1^2 \right) \quad &\text{in}& \quad N=2 \\
C \left(\rho^2 + \lambda^2 \| h_\lambda \|_\infty^2 \right) \quad &\text{in}& \quad N=3
\end{array}\right. ,
\end{eqnarray}
where we have employed Young inequality in the last step.
Lets define now, for any $\varphi \in \mathcal{B}$, the nonlinear operator
\begin{eqnarray}\nonumber
\mathcal{T}: \mathcal{B} &\longrightarrow& \mathcal{B} \\ \nonumber
\varphi &\longmapsto& \mathcal{T}(\varphi)=w,
\end{eqnarray}
where $w$ is the unique solution to the singular boundary value problem
$$
\left\{ \begin{array}{rcl}
-w'' + (N-2)w' + (N-1)w &=& \frac{N-1}{2} e^{-t} \varphi^2 \\
&+& \lambda e^{(N-3)t} \int_0^{e^{-t}} g(s) \, ds, \\
w(0)' -  (N-1)w(0)=w(\infty) &=& 0.
\end{array}\right.
$$
For sufficiently small $\rho$ and $|\lambda|$ the operator $\mathcal{T}$ is both well defined,
i.~e. bounded in the ball $\mathcal{B}$, as a consequence of~\eqref{ball2ball2},
and contractive as a consequence of~\eqref{contract2}.
The existence of a unique fixed point
of $\mathcal{T}$, which in turn is a solution to problem~\eqref{navier} with $N=2,3$, follows from the Banach fixed point theorem.
\end{proof}

\begin{remark}
The origin of $\tilde{H}^1_\mu$ is contained in the ball $\mathcal{B}$ given that $|\lambda|$ is sufficiently small.
\end{remark}


\section{Regularity results for $k=2$ and $N \in \{2,3\}$}

\subsection{Dirichlet problem}

We have already proven the existence of at least one weak solution to the singular boundary value problem
\begin{equation}
\label{dirichlet3}
\left\{ \begin{array}{rcl}
-w'' + (N-2)w' + (N-1)w &=& \frac{N-1}{2} e^{-t} w^2 \\
&+& \lambda e^{(N-3)t} \int_0^{e^{-t}} g(s) \, ds, \\
w( 0)=w(\infty) &=& 0,
\end{array}\right.
\end{equation}
in the previous section. In the present subsection we prove higher regularity for this type of solution.

\begin{theorem}\label{regu}
Weak solutions to this singular boundary value problem actually belong to the space
$C^2(\mathbb{R}_+) \cap W^{2,p}(\mathbb{R}_+)$ $\forall \, 2 \le p \le \infty$ if $N=2$ and
$C^2(\mathbb{R}_+) \cap W^{2,p}(\mathbb{R}_+,e^{-t}dt)$ $\forall \, 2 \le p \le \infty$ if $N=3$.
\end{theorem}

\begin{proof}
First note that
$$
\| u \|_2 \equiv \sup_{\| \psi \|_2 = 1} \langle \psi, u \rangle \quad \forall \quad u \in L^2(\mathbb{R}_+),
$$
where $\langle \cdot, \cdot \rangle$ denotes the scalar product in $L^2(\mathbb{R}_+)$, and
$$
\| u \|_2' \equiv \sup_{\| \psi \|_2' = 1} \langle \psi, u \rangle' \quad \forall \quad u \in L^2(\mathbb{R}_+, e^{-t} \, dt),
$$
where $\| \cdot \|_2'$ and $\langle \cdot , \cdot \rangle'$ denote the norm and scalar product in $L^2(\mathbb{R}_+, e^{-t} \, dt)$ respectively.

\textsc{Step 1: $N=2$.} Take the $\sup_{\| \psi \|_2 = 1} \langle \psi, \cdot \rangle$ on both sides of the equation to find
\begin{eqnarray}\nonumber
\|w''\|_2 &=& \sup_{\| \psi \|_2 = 1} \langle \psi, -w'' \rangle \le \sup_{\| \psi \|_2 = 1} \langle \psi, -w \rangle
+ \frac12 \sup_{\| \psi \|_2 = 1} \langle \psi, e^{-t} w^2 \rangle \\ \nonumber & &
+ |\lambda| \sup_{\| \psi \|_2 = 1} \left\langle \psi, e^{-t} \int_0^{e^{-t}} g(s) \, ds \right\rangle \\ \nonumber
&\le& \| w \|_2 + \frac12 \| w \|_4^2 + |\lambda| \left\| e^{-t} \int_0^{e^{-t}} g(s) \, ds \right\|_2,
\end{eqnarray}
where we have employed the triangle inequality in the first step and H\"older inequality in the second. Note the right hand side
is finite since $w \in L^4(\mathbb{R}_+)$ and $e^{-t} \int_0^{e^{-t}} g(s) \, ds \in L^2(\mathbb{R}_+)$ by interpolation.
So $w \in H^2(\mathbb{R}_+)$. On the other hand we have
$$
-w'' = -w + \frac12 e^{-t} w^2 + \lambda e^{-t} \int_0^{e^{-t}} g(s) \, ds,
$$
with all terms on the right hand side belonging to $BC(\mathbb{R}_+)$. The result follows by interpolation.

\textsc{Step 2: $N=3$.} Now take the $\sup_{\| \psi \|_2' = 1} \langle \psi, \cdot \rangle'$ on both sides of our equation to get
\begin{eqnarray}\nonumber
\|w''\|_2' &=& \sup_{\| \psi \|_2' = 1} \langle \psi, -w'' \rangle'
\le \sup_{\| \psi \|_2' = 1} \langle \psi, -w' \rangle' + \sup_{\| \psi \|_2' = 1} \langle \psi, -2w \rangle'
\\ \nonumber & & + \sup_{\| \psi \|_2' = 1} \langle \psi, e^{-t} w^2 \rangle'
+ |\lambda| \sup_{\| \psi \|_2' = 1} \left\langle \psi, \int_0^{e^{-t}} g(s) \, ds \right\rangle' \\ \nonumber
&=& \| w' \|_2' + 2 \| w \|_2' + \sup_{\| \psi \|_2' = 1} \int_0^\infty e^{-2t} w(t)^2 \, \psi(t) \, dt \\ \nonumber
& & + |\lambda| \sup_{\| \psi \|_2' = 1} \int_0^\infty e^{-t} \psi(t) \left[ \int_0^{e^{-t}} g(s) \, ds \right] dt ,
\end{eqnarray}
where we have used the triangle inequality. From the proof of Lemma~\ref{linapp} it is clear that $w e^{-t} \in L^\infty(\mathbb{R}_+)$.
Therefore
$$
\left| \int_0^\infty e^{-2t} w^2 \, \psi \, dt \right| \le \left( \int_0^\infty e^{-3t} w^4 \, dt \right)^{1/2} \le \| e^{-t} w \|_\infty \| w \|_2',
$$
where we have used H\"older inequality in both steps, so it is a bounded quantity.
Finally
\begin{eqnarray}\nonumber
\left| \int_0^\infty e^{-t} \psi(t) \left[ \int_0^{e^{-t}} g(s) \, ds \right] dt \right| &\le& \\ \nonumber
\left( \int_0^\infty e^{-t} \left| \int_0^{e^{-t}} g(s) \, ds \right|^2 dt \right)^{1/2} &\le&
\left\| \int_0^{e^{-t}} g(s) \, ds \right\|_\infty,
\end{eqnarray}
after the iterative application of H\"older inequality.
Therefore $w \in H^2(\mathbb{R}_+, e^{-t} \, dt)$ and, noting that this norm is equivalent to the standard $H^2(\mathbb{R}_+)$ norm on compact intervals
of $\mathbb{R}_+$ and invoking the corresponding Sobolev embedding, we find $w \in C^1(\mathbb{R}_+)$.
Going back to the equation
$$
-w'' = -w' - 2w + e^{-t} w^2 + \lambda \int_0^{e^{-t}} g(s) \, ds,
$$
we can check that all terms in the right hand side are continuous and thus $w \in C^2(\mathbb{R}_+)$. The statement follows again by interpolation.
\end{proof}

\begin{remark}\label{regure}
Taking into account that the solutions fulfill the boundary conditions it follows that $w \in BC^2(\mathbb{R}_+)$ for both $N=2,3$.
Also, invoking Remark~\ref{acdat} and arguing like the proof of the previous theorem it is clear
that $w \in AC_{\mathrm{loc}}^2(\mathbb{R}_+)$ for both $N=2,3$.
\end{remark}

\subsection{Navier problem}
\label{hrnavier}

The higher regularity of the solutions to problem
\begin{equation}
\label{navier2}
\left\{ \begin{array}{rcl}
-w'' + (N-2)w' + (N-1)w &=& \frac{N-1}{2} e^{-t} w^2 \\
&+& \lambda e^{(N-3)t} \int_0^{e^{-t}} g(s) \, ds, \\
w'(0) - (N-1) w(0) = w(+\infty) &=& 0,
\end{array}\right.
\end{equation}
follows analogously to the result in the previous section.

\begin{theorem}\label{regu2}
Weak solutions to this singular boundary value problem actually belong to the space
$C^2(\mathbb{R}_+) \cap W^{2,p}(\mathbb{R}_+)$ $\forall \, 2 \le p \le \infty$ if $N=2$ and
$C^2(\mathbb{R}_+) \cap W^{2,p}(\mathbb{R}_+,e^{-t}dt)$ $\forall \, 2 \le p \le \infty$ if $N=3$.
\end{theorem}

\begin{proof}
\textsc{Step 1: $N=2$.} Following the previous section
we take the $\sup_{\| \psi \|_2 = 1} \langle \psi, \cdot \rangle$ on both sides of the equation to find
\begin{eqnarray}\nonumber
\|w''\|_2 &=& \sup_{\| \psi \|_2 = 1} \langle \psi, -w'' \rangle \le \sup_{\| \psi \|_2 = 1} \langle \psi, -w \rangle
+ \frac12 \sup_{\| \psi \|_2 = 1} \langle \psi, e^{-t} w^2 \rangle \\ \nonumber & &
+ |\lambda| \sup_{\| \psi \|_2 = 1} \left\langle \psi, e^{-t} \int_0^{e^{-t}} g(s) \, ds \right\rangle \\ \nonumber
&\le& \| w \|_2 + \frac12 \| w \|_4^2 + |\lambda| \left\| e^{-t} \int_0^{e^{-t}} g(s) \, ds \right\|_2,
\end{eqnarray}
after the use of the triangle and H\"older inequalities respectively. The right hand side
is finite because $w \in L^4(\mathbb{R}_+)$ and $e^{-t} \int_0^{e^{-t}} g(s) \, ds \in L^2(\mathbb{R}_+)$ as can be deduced by interpolation.
Then $w \in H^2(\mathbb{R}_+)$ and
$$
-w'' = -w + \frac12 e^{-t} w^2 + \lambda e^{-t} \int_0^{e^{-t}} g(s) \, ds,
$$
where all the summands on the right hand side are uniformly continuous, so the statement follows by interpolation.

\textsc{Step 2: $N=3$.} Again we take the $\sup_{\| \psi \|_2' = 1} \langle \psi, \cdot \rangle'$ on both sides of our
differential equation to find
\begin{eqnarray}\nonumber
\|w''\|_2' &=& \sup_{\| \psi \|_2' = 1} \langle \psi, -w'' \rangle'
\le \sup_{\| \psi \|_2' = 1} \langle \psi, -w' \rangle' + \sup_{\| \psi \|_2' = 1} \langle \psi, -2w \rangle'
\\ \nonumber & & + \sup_{\| \psi \|_2' = 1} \langle \psi, e^{-t} w^2 \rangle'
+ |\lambda| \sup_{\| \psi \|_2' = 1} \left\langle \psi, \int_0^{e^{-t}} g(s) \, ds \right\rangle' \\ \nonumber
&=& \| w' \|_2' + 2 \| w \|_2' + \sup_{\| \psi \|_2' = 1} \int_0^\infty e^{-2t} w(t)^2 \, \psi(t) \, dt \\ \nonumber
& & + |\lambda| \sup_{\| \psi \|_2' = 1} \int_0^\infty e^{-t} \psi(t) \left[ \int_0^{e^{-t}} g(s) \, ds \right] dt ,
\end{eqnarray}
after invoking the triangle inequality. The proof of Lemma~\ref{linapp2} implies that $w e^{-t} \in L^\infty(\mathbb{R}_+)$,
and thus
$$
\left| \int_0^\infty e^{-2t} w^2 \, \psi \, dt \right| \le \left( \int_0^\infty e^{-3t} w^4 \, dt \right)^{1/2} \le \| e^{-t} w \|_\infty \| w \|_2',
$$
after invoking twice H\"older inequality.
Again a double application of H\"older inequality leads to
\begin{eqnarray}\nonumber
\left| \int_0^\infty e^{-t} \psi(t) \left[ \int_0^{e^{-t}} g(s) \, ds \right] dt \right| &\le& \\ \nonumber
\left( \int_0^\infty e^{-t} \left| \int_0^{e^{-t}} g(s) \, ds \right|^2 dt \right)^{1/2} &\le&
\left\| \int_0^{e^{-t}} g(s) \, ds \right\|_\infty.
\end{eqnarray}
So far we have proven $w \in H^2(\mathbb{R}_+, e^{-t} \, dt)$,
a norm that is equivalent to the $H^2(\mathbb{R}_+)$ norm on compact intervals
of $\mathbb{R}_+$, and thus a suitable Sobolev embedding shows $w \in C^1(\mathbb{R}_+)$.
If we consider again our equation
$$
-w'' = -w' - 2w + e^{-t} w^2 + \lambda \int_0^{e^{-t}} g(s) \, ds,
$$
it is clear that all the summands in the right hand side are continuous and consequently $w \in C^2(\mathbb{R}_+)$.
We conclude by interpolation .
\end{proof}

\begin{remark}\label{regure2}
Appreciating the obvious fact that the solutions to our singular boundary value problem
obey the boundary conditions then we necessarily have $w \in BC^2(\mathbb{R}_+)$ for both $N=2,3$.
Moreover, an argument akin to that in the proof of Theorem~\ref{regu2}, together with Remark~\ref{acdat}
shows that $w \in AC_{\mathrm{loc}}^2(\mathbb{R}_+)$ for both $N=2,3$.
\end{remark}

\section{Existence via upper and lower solutions}

\subsection{Dirichlet problem}

In this subsection, we are going to present an alternative approach for the existence of solutions of the Dirichlet problem~\eqref{dirichlet2} with $N=2,3$ based on the method of upper and lower solutions. For basic definitions, we refer to~\cite{Sch67}. For convenience, we define the operators
\begin{eqnarray*}
\mathcal{L} w & := & -w''+(N-2)w'+(N-1)w, \\
{\mathcal N} w & := & \frac{N-1}{2} e^{-t}w^2,
\end{eqnarray*}
and recall the definition
$$
h_1(t)=e^{(N-3)t} \int_0^{e^{-t}} g(s) \, ds.
$$

Lemma~\ref{linapp} provides a well-defined inverse operator $\mathcal{L}^{-1}$, which will be used below. The concrete expression of such operator is
\begin{eqnarray*}
\mathcal{L}^{-1} h(t) &=& -\frac{e^{-t}}{N} \int_0^\infty e^{(1-N)s} \, h(s) \, ds + \frac{e^{-t}}{N} \int_0^t e^s \, h(s) \, ds \\
&+& \frac{e^{(N-1)t}}{N} \int_t^\infty e^{(1-N)s} \, h(s) \, ds.
\end{eqnarray*}

\begin{lemma}\label{lower}
The function $\alpha=\lambda \mathcal{L}^{-1}h_1(t)$ is a lower solution of~\eqref{dirichlet2}.
\end{lemma}

\begin{proof}
By the positivity of operator $\mathcal N$, we have
$$
\mathcal{L} \alpha=\lambda h_1(t)\leq {\mathcal N}\alpha+\lambda h_1(t).
$$
\end{proof}

Note that, as defined above, $\alpha(0)=0=\alpha(+\infty)$.

\begin{lemma}\label{upper}
Take a constant $\beta>0$ such that
\begin{equation}\label{cond1}
(N-1)\beta\geq \frac{N-1}{2}\beta^2+\lambda h_1(t), \qquad\mbox{ for all} \; t>0.
\end{equation}
Then,  $\beta$ is an upper solution of~\eqref{dirichlet2} and $\alpha(t)<\beta$ for all $t>0$.
\end{lemma}

\begin{proof}
Trivially, $\mathcal{L} \beta= (N-1)\beta$, then condition \eqref{cond1} reads
$$
\mathcal{L}\beta\geq  \frac{N-1}{2}\beta^2+\lambda h_1(t)\geq  \frac{N-1}{2}e^{-t}\beta^2+\lambda h_1(t),
$$
that is, $\beta$  is an upper solution. Besides, the operator $\mathcal{L}^{-1}$ is linear and positive, hence
$$
\beta\geq \frac{N-1}{2} \mathcal{L}^{-1}\beta^2+\lambda \mathcal{L}^{-1}h_1(t)>\alpha(t).
$$
\end{proof}

Note that condition~\eqref{cond1} holds for any $0<\beta<2$ if $\lambda$ is small enough, depending on $\beta$. The optimal choice is $\beta=1$, for which condition~\eqref{cond1} reads
\begin{equation}\label{cond2}
\lambda h_1(t)\leq\frac{N-1}{2} , \qquad\mbox{ for all} \; t>0.
\end{equation}

\begin{theorem}\label{th_exist}
Under the hypothesis~\eqref{cond2}, problem~\eqref{dirichlet2} has at least one solution.
\end{theorem}

\begin{proof}
Lemmas~\ref{lower} and~\ref{upper} provide a couple of well-ordered lower and upper solutions. The upper solution is $\beta=1$.  Then, Theorem 4.1
of~\cite{Sch67} assures the existence of a solution $w\in  BC^2(\mathbb{R}_+)$ of the equation such that
$$
\alpha(t)\leq w(t)\leq \beta\equiv1 \qquad\mbox{ for all} \; t>0
$$
and $w(0)=0$. It remains to check the condition at $+\infty$. To this aim, let us observe that
$$
w(t)=  \frac{N-1}{2}\mathcal{L}^{-1} e^{-t}w^2+\lambda \mathcal{L}^{-1} h_1(t)\leq  \frac{N-1}{2} \mathcal{L}^{-1} e^{-t}+\lambda \mathcal{L}^{-1} h_1(t).
$$
From the properties of $\mathcal{L}^{-1}$, it is trivial that $w(+\infty)=0$.
\end{proof}

\begin{remark}
Note that Theorem~\ref{th_exist} is of a different nature that the results presented on Section~3, that are based on the contraction principle. We lose the information about uniqueness near zero, but on the contrary condition~\eqref{cond2} provides a global bound for the interval of $\lambda$ where the problem is solvable. In particular, observe that such condition is void if $h_1(t) \le 0$ for $t>0$. In section~6 we present a complementary result.
\end{remark}

\subsection{Navier problem}

In this case we cannot directly use reference~\cite{Sch67}. Instead we will build the iterative procedure typical of the method
of upper and lower solutions directly~\cite{subsuper}. We will borrow the notation from the previous subsection and the explicit formula for
$\mathcal{L}^{-1}$ from subsection~\ref{sec:exnav}, where this integral operator is shown to be well-defined.

\begin{theorem}\label{th_exist2}
Under the hypothesis~\eqref{cond2}, problem~\eqref{navier} has at least one solution.
\end{theorem}

\begin{proof}
\textsc{Step 1: Existence of lower solution.}
We claim the function $\alpha(t)=\lambda \mathcal{L}^{-1}h_1(t)$ is a lower solution of~\eqref{navier}.
Define
\begin{equation}\label{aproxs}
\left\{ \begin{array}{rcl}
-w_k'' + (N-2)w_k' + (N-1)w_k &=& \frac{N-1}{2} e^{-t} w_{k-1}^2 \\
&+& \lambda e^{(N-3)t} \int_0^{e^{-t}} g(s) \, ds, \\
w_k'(0) - (N-1) w_k(0) = w_k(+\infty) &=& 0,
\end{array}\right.
\end{equation}
for $k \in \mathbb{N}$ where $w_0 \equiv \alpha$. It is clear that
$$
w_1= \mathcal{L}^{-1} \mathcal{N} \alpha + \lambda \mathcal{L}^{-1} h_1 \ge \lambda \mathcal{L}^{-1} h_1 = \alpha,
$$
due to the positivity of $\mathcal{L}$ and $\mathcal{N}$. Now assume $w_k \ge w_{k-1}$ and compute
\begin{eqnarray}\nonumber
w_{k+1}(t)-w_k(t) &=& \mathcal{L}^{-1} \mathcal{N} w_k - \mathcal{L}^{-1} \mathcal{N} w_{k-1} \\ \nonumber
&=& \frac{e^{-t}(N-1)}{2N} \int_0^t \left[ w_k(s)^2 - w_{k-1}(s)^2 \right] ds + \\ \nonumber & &
\frac{e^{(N-1)t}(N-1)}{2N} \int_t^\infty \! e^{-N s} \left[ w_k(s)^2 - w_{k-1}(s)^2 \right] ds \ge 0.
\end{eqnarray}
Therefore by induction we conclude
$$
\alpha \le w_1 \le \cdots \le w_{k-1} \le w_k \le \cdots .
$$

\textsc{Step 2: Existence of upper solution.}
For any constant $\beta>0$ such that
\begin{equation}\nonumber
(N-1)\beta \ge \frac{N-1}{2}\beta^2+\lambda h_1(t), \qquad\mbox{ for all} \; t>0,
\end{equation}
we have $\alpha(t)<\beta$ for all $t>0$
(note that the set of all constants $\beta$ fulfilling this inequality is nonempty for $\lambda$ small enough). Indeed
$$
\beta \ge \frac{N-1}{2} \mathcal{L}^{-1}\beta^2+\lambda \mathcal{L}^{-1}h_1(t)>\alpha(t),
$$
since $\mathcal{L}\beta=(N-1)\beta$.
Now assume $w_k \le \beta$ and compute
\begin{eqnarray}\nonumber
\beta - w_{k+1} &\ge& \mathcal{L}^{-1} \mathcal{N} \beta - \mathcal{L}^{-1} \mathcal{N} w_{k-1} \\ \nonumber
&=& \frac{e^{-t}(N-1)}{2N} \int_0^t \left[ \beta^2 - w_{k}(s)^2 \right] ds + \\ \nonumber & &
\frac{e^{(N-1)t}(N-1)}{2N} \int_t^\infty \! e^{-N s} \left[ \beta^2 - w_{k}(s)^2 \right] ds \ge 0.
\end{eqnarray}
Therefore by induction we find
\begin{equation}\label{sandwich}
\alpha \le w_1 \le \cdots \le w_{k-1} \le w_k \le \cdots \le \beta.
\end{equation}
Note again that the optimal choice is $\beta =1$, for which condition~\eqref{cond1} translates into~\eqref{cond2}.

\textsc{Step 3: Convergence.}
We define
$$
w(t):= \lim_{k \to \infty} w_k(t),
$$
which is well defined by monotonicity and boundedness of the sequence, see~\eqref{sandwich}.
Moreover, we may invoke the dominated convergence theorem to see $w_k \to w$ in $L^2(\mathbb{R}_+,e^{-t}dt)$.
By the proof of Lemma~\ref{linapp2} we know the set $w_k$ is uniformly bounded in $\tilde{H}^1_\mu$ and therefore it possesses a weakly
convergent subsequence in this space. Therefore we can safely take the limit $k \to \infty$ in equation~\eqref{aproxs}.
One can see that the boundary condition at $+\infty$ is obeyed by $w$ exactly as in the proof of Theorem~\ref{th_exist},
while the boundary condition at the origin follows from the formulas
$$
w_{k}(0), \, w_k'(0) \propto \int_0^\infty e^{-Ns} w_{k-1}(s)^2 \, ds + \mathcal{C},
$$
where the proportionality constant as well as the constant $\mathcal{C}$ are $k-$independent.
Finally, the higher regularity is obtained as in subsection~\ref{hrnavier}.
\end{proof}

\section{Non-existence results for $k=2$ and $N \in \{2,3\}$}

This section is devoted to prove the non-existence of weak solutions in the same sense as in section~\ref{exisrel},
i.~e. in $H^1_\mu$ and $\tilde{H}^1_\mu$ respectively, to the singular boundary value problems under consideration.
The key assumption is a large enough $\lambda>0$.

\subsection{Dirichlet problem}

\begin{lemma}\label{initval}
The initial value problem
\begin{equation}\nonumber
\left\{ \begin{array}{rcl}
-w'' + (N-2)w' + (N-1)w &=& \frac{N-1}{2} e^{-t} w^2 \\
&+& \lambda e^{(N-3)t} \int_0^{e^{-t}} g(s) \, ds, \\
w(0)=0, \qquad w'(0) &=& w_0.
\end{array}\right.
\end{equation}
has a solution that fulfills $\lim_{t \to \infty} w(t)=0$ only if $w_0 \le \int_0^\infty e^{(1-N)s} h_\lambda(s) ds$.
\end{lemma}

\begin{proof}
Note that the solution to
\begin{equation}\nonumber
\left\{ \begin{array}{rcl}
-v'' + (N-2)v' + (N-1)v &=& h_\lambda(t) \\
v(0)=0, \qquad v'(0) &=& w_0,
\end{array}\right.
\end{equation}
reads
$$
v(t)= \frac{e^{-t}}{N}(e^{Nt}-1)w_0 + \frac{e^{-t}}{N} \int_0^t e^s h_\lambda(s) ds -\frac{e^{(N-1)t}}{N} \int_0^t e^{(1-N)s} h_\lambda(s) ds,
$$
where
$$
h_\lambda(t) = \lambda e^{(N-3)t} \int_0^{e^{-t}} g(s) \, ds.
$$
Then $w(t) \ge v(t)$ and $\lim_{t \to \infty} v(t)=+\infty$ provided
$$
w_0 > \int_0^\infty e^{(1-N)s} h_\lambda(s) ds.
$$
\end{proof}

\begin{remark}
The initial condition $w'(0)$ is interpreted in the same sense as in subsection~\ref{sec:exnav}.
\end{remark}

\begin{theorem}\label{nondir}
The singular boundary value problem
\begin{equation}\nonumber
\left\{ \begin{array}{rcl}
-w'' + (N-2)w' + (N-1)w &=& \frac{N-1}{2} e^{-t} w^2 \\
&+& \lambda e^{(N-3)t} \int_0^{e^{-t}} g(s) \, ds, \\
w(0)=w(\infty) &=& 0.
\end{array}\right.
\end{equation}
has no solutions provided $g \ge 0$, $ \operatorname{ess} \sup g >0$ and $\lambda>0$ is large enough.
\end{theorem}

\begin{proof}
We can rewrite this boundary value problem as
\begin{equation}\nonumber
w(t) = \frac{e^{-t}}{N} \int_0^t e^s ( 1 - e^{-N s} ) \, h(s) \, ds
+ \frac{e^{-t}}{N} (e^{N t}-1) \int_t^\infty e^{(1-N)s} \, h(s) \, ds,
\end{equation}
by Lemma~\ref{linapp}, where $h(t)=h^*(t)+h_\lambda(t)$ and
\begin{eqnarray}\nonumber
h^* &=& \frac{N-1}{2} e^{-t} w^2, \\ \nonumber
h_\lambda &=& \lambda e^{(N-3)t} \int_0^{e^{-t}} g(s) \, ds,
\end{eqnarray}
see subsection~\ref{sec:exdir}. From this formula it is clear that $w(t)>0 \, \forall \, t >0$ under the hypotheses of the statement.
Moreover we have
\begin{eqnarray}\nonumber
w(t) &\ge& \frac{e^{-t}}{N} \int_0^t e^s ( 1 - e^{-N s} ) \, h_\lambda(s) \, ds \\ \nonumber
& & + \frac{e^{-t}}{N} (e^{N t}-1) \int_t^\infty e^{(1-N)s} \, h_\lambda(s) \, ds \\ \nonumber
&=:& \lambda \, \tilde{h}(t),
\end{eqnarray}
and therefore $w(t) \ge \lambda \, \tilde{h}(t)$ for a function $\tilde{h}(t)>0 \, \forall \, t >0$,
$\lim_{t \to \infty} \tilde{h}(t)=0$ and $\tilde{h}(0)=0$.
This implies
\begin{eqnarray}\nonumber
w(t) &\ge& \frac{e^{-t}}{N} \int_0^t e^s ( 1 - e^{-N s} ) \, h^*(s) \, ds \\ \nonumber
& & + \frac{e^{-t}}{N} (e^{N t}-1) \int_t^\infty e^{(1-N)s} \, h^*(s) \, ds \\ \nonumber
&\ge& \lambda^2 \frac{(N-1)e^{-t}}{2N} \int_0^t ( 1 - e^{-N s} ) \, \tilde{h}^2(s) \, ds \\ \nonumber
& & + \lambda^2 \frac{(N-1)e^{-t}}{2N} (e^{N t}-1) \int_t^\infty e^{-Ns} \, \tilde{h}^2(s) \, ds.
\end{eqnarray}
A straightforward calculation yields $w'(0) \ge \lambda^2 \, C$ for
$$
C = \frac{N-1}{2} \int_0^\infty e^{-N s} \, \tilde{h}^2(s) \, ds >0.
$$
Note that, under the hypothesis of the statement, Lemma~\ref{initval} implies that a necessary condition for the
existence of solution is $w'(0) \le C' \lambda$ for some $C'>0$. The desired conclusion follows as a consequence of
these two facts.
\end{proof}

\subsection{Navier problem}

\begin{lemma}\label{initval2}
The initial value problem
\begin{equation}\nonumber
\left\{ \begin{array}{rcl}
-w'' + (N-2)w' + (N-1)w &=& \frac{N-1}{2} e^{-t} w^2 \\
&+& \lambda e^{(N-3)t} \int_0^{e^{-t}} g(s) \, ds, \\
w(0)=w_0, \qquad w'(0)-(N-1)w(0) &=& 0.
\end{array}\right.
\end{equation}
has a solution that fulfills $\lim_{t \to \infty} w(t)=0$ only if $w_0 \le N^{-1} \int_0^\infty e^{(1-N)s} h_\lambda(s) ds$.
\end{lemma}

\begin{proof}
Notice that the solution to
\begin{equation}\nonumber
\left\{ \begin{array}{rcl}
-v'' + (N-2)v' + (N-1)v &=& h_\lambda(t) \\
v(0)=w_0, \qquad v'(0)-(N-1)v(0) &=& 0,
\end{array}\right.
\end{equation}
is given by the formula
$$
v(t)= e^{(N-1)t} w_0 + \frac{e^{-t}}{N} \int_0^t e^s h_\lambda(s) ds -\frac{e^{(N-1)t}}{N} \int_0^t e^{(1-N)s} h_\lambda(s) ds,
$$
where
$$
h_\lambda(t) = \lambda e^{(N-3)t} \int_0^{e^{-t}} g(s) \, ds.
$$
Therefore $w(t) \ge v(t)$ and $\lim_{t \to \infty} v(t)=+\infty$ if
$$
w_0 > \frac1N \int_0^\infty e^{(1-N)s} h_\lambda(s) ds.
$$
\end{proof}

\begin{theorem}\label{nonnav}
The singular boundary value problem
\begin{equation}\nonumber
\left\{ \begin{array}{rcl}
-w'' + (N-2)w' + (N-1)w &=& \frac{N-1}{2} e^{-t} w^2 \\
&+& \lambda e^{(N-3)t} \int_0^{e^{-t}} g(s) \, ds, \\
w'(0)-(N-1)w(0)=w(\infty) &=& 0.
\end{array}\right.
\end{equation}
has no solutions provided $g \ge 0$, $ \operatorname{ess} \sup g >0$ and $\lambda>0$ is large enough.
\end{theorem}

\begin{proof}
Again invoking the Green function we find
\begin{equation}\nonumber
w(t) = \frac{e^{-t}}{N} \int_0^t e^s \, h(s) \, ds
+ \frac{e^{(N-1)t}}{N} \int_t^\infty e^{(1-N)s} \, h(s) \, ds,
\end{equation}
by Lemma~\ref{linapp2}, where $h(t)=h^*(t)+h_\lambda(t)$ and
\begin{eqnarray}\nonumber
h^* &=& \frac{N-1}{2} e^{-t} w^2, \\ \nonumber
h_\lambda &=& \lambda e^{(N-3)t} \int_0^{e^{-t}} g(s) \, ds,
\end{eqnarray}
see subsection~\ref{sec:exnav}. From here it is obvious that $w(t)>0 \, \forall \, t >0$ under the hypotheses of the statement.
Furthermore we know
\begin{equation}\nonumber
w(t) \ge \frac{e^{-t}}{N} \int_0^t e^s \, h_\lambda(s) \, ds
+ \frac{e^{(N-1)t}}{N} \int_t^\infty e^{(1-N)s} \, h_\lambda(s) \, ds
=: \lambda \, \tilde{h}(t),
\end{equation}
and consequently $w(t) \ge \lambda \, \tilde{h}(t)$, where $\tilde{h}(t)>0 \, \forall \, t >0$ and
$\lim_{t \to \infty} \tilde{h}(t)=0$.
In turn this implies
\begin{eqnarray}\nonumber
w(t) &\ge& \frac{e^{-t}}{N} \int_0^t e^s \, h^*(s) \, ds
+ \frac{e^{(N-1)t}}{N} \int_t^\infty e^{(1-N)s} \, h^*(s) \, ds \\ \nonumber
&\ge& \lambda^2 \frac{N-1}{2N} e^{-t} \int_0^t \tilde{h}^2(s) \, ds
+ \lambda^2 \frac{N-1}{2N} e^{(N-1)t} \int_t^\infty e^{-N s} \, \tilde{h}^2(s) \, ds.
\end{eqnarray}
Evaluating this inequality at $t=0$ yields $w(0) \ge \lambda^2 \, C$ for
$$
C = \frac{N-1}{2N} \int_0^\infty e^{-N s} \, \tilde{h}^2(s) \, ds >0.
$$
Notice that, under the hypothesis of the statement, Lemma~\ref{initval2} says that a solution only exists if
$w(0) \le C' \lambda$ for a positive $C'$. The statement is a consequence of these two inequalities.
\end{proof}

\section{Conclusions}

The elliptic problem we have considered in this work, equation~\eqref{rkhessian}, describes the stationary solutions of a model in the theory
of non-equilibrium phase transitions~\cite{escudero,escudero2}. One of the most important models in the theory of equilibrium phase transitions
is the Ginzburg-Landau equation~\cite{binney,chaikin}.
The number of stationary solutions to this equation changes from one to three as some parameter varies, a fact
that is related to the presence of a phase transition. The existence theory for equations like~\eqref{rkhessian} is far less obvious than for
equations like the Ginzburg-Landau one, and it could be related to the presence of non-equilibrium phase transitions. Still, we have found that
there exists at least one solution to the boundary value problem under consideration provided the parameter $\lambda$ is negative, zero or a small
positive real number. On the other hand, if this parameter is positive and large enough, no solutions exist. Whether this signals the presence of
a phase transition for a critical value of $\lambda$, is a subject still to be investigated. Although our present results are promising in this
direction, more work has to be done in order to certify this possibility.

\vskip5mm
\noindent
{\footnotesize
Carlos Escudero\par\noindent
Departamento de Matem\'aticas\par\noindent
Universidad Aut\'onoma de Madrid\par\noindent
{\tt carlos.escudero@uam.es}\par\vskip1mm\noindent
\& \par\vskip1mm\noindent
Pedro J. Torres\par\noindent
Departamento de Matem\'atica Aplicada\par\noindent
Universidad de Granada\par\noindent
{\tt ptorres@ugr.es}\par\vskip1mm\noindent
}
\end{document}